\documentclass[reqno]{amsart}
\usepackage[scale=0.75, centering, headheight=14pt]{geometry}
\usepackage[latin1]{inputenc}
\usepackage[T1]{fontenc}
\usepackage{lmodern}
\usepackage[english]{babel}
\usepackage{esint}
\usepackage{mathtools}

\usepackage{xcolor,amsmath,mathrsfs,amssymb,accents,graphicx,epstopdf, amsthm}
\usepackage{ bbold }
\usepackage{comment}
\usepackage{float}
\usepackage{subfig}
\usepackage{tikz}

\usepackage{tikz}
\usepackage{bbm}
\usepackage{amsmath,amssymb,amsfonts,amsthm}
\usepackage{mathtools,accents}
\usepackage{mathrsfs}
\usepackage{xfrac}
\usepackage{array} 
\usepackage{aliascnt}
\usepackage{booktabs} 
\usepackage{array} 

\usepackage{verbatim} 
\usepackage{subfig} 

\usepackage{mathrsfs, dsfont}
\usepackage{amssymb}
\usepackage{amsthm}
\usepackage{amsmath,amsfonts,amssymb,esint}
\usepackage{graphics,color}
\usepackage{enumerate}
\usepackage{mathtools,centernot}

\usepackage{microtype}
\usepackage{paralist} 
\usepackage{cases}
\usepackage[initials]{amsrefs}
\allowdisplaybreaks

\usepackage{braket}
\usepackage{bm}

\usepackage[citecolor=blue,colorlinks]{hyperref}
\addto\extrasenglish{}

\usepackage{enumerate}
\usepackage{xcolor}
\usepackage{aliascnt}

\numberwithin{equation}{section}

\newtheorem{theorem}{Theorem}[section]
\newtheorem{question}{Question}
\newtheorem{lemma}[theorem]{Lemma}
\newtheorem{definition}[theorem]{Definition}
\newtheorem{remark}[theorem]{Remark}
\newtheorem{proposition}[theorem]{Proposition}
\newtheorem{corollary}[theorem]{Corollary}

\usepackage{enumitem,xparse}
\newlist{Claim}{description}{2}
\setlist[Claim]{labelindent=2em,leftmargin=*}
\newif\ifInsideClaim
\newcounter{claim}[theorem]
\newcounter{cclaim}[claim]

\let\originalqedsymbol\qedsymbol

\newcommand{\Q}{\mathbb{Q}}
\newcommand{\N}{\mathbb{N}}
\newcommand{\R}{\mathbb{R}}

\newcommand{\Z}{\mathbb{Z}}
\newcommand{\T}{\mathbb{T}}
\newcommand{\Leb}[1]{{\mathscr L}^{#1}} 

\renewcommand{\P}{\mathbb{P}}

\newcommand{\x}{\times}

\renewcommand{\a}{\alpha}

\newcommand{\e}{\varepsilon}

\newcommand{\s}{\sigma}

\renewcommand\div{\operatorname{div}}

\newcommand{\supp}{\operatorname{supp}}

\newcommand{\vo}{\vec{o}\@ifnextchar{^}{\,}{}}

 \newcommand{\bb}{{\mbox{\boldmath$b$}}}
  
 \newcommand{\cc}{{\mbox{\boldmath$c$}}}

 \newcommand{\uu}{{\mbox{\boldmath$u$}}}
 
 \newcommand{\ww}{{\mbox{\boldmath$w$}}}

 \newcommand{\tauV}{{\kern-3pt\tau}}

 \newcommand{\XX}{{\mbox{\boldmath$X$}}}

 \newcommand{\oVVVk}{\overline{\mbox{\boldmath$V$}}\kern-3pt}
 \newcommand{\tVVVk}{\tilde{\mbox{\boldmath$V$}}\kern-3pt}

 \newcommand{\ddelta}{{\mbox{\boldmath$\delta$}}}

 \newcommand{\mmu}{{\mbox{\boldmath$\mu$}}}
 \newcommand{\nnu}{{\mbox{\boldmath$\nu$}}}

 \newcommand{\eeta}{{\mbox{\boldmath$\eta$}}}

	\title{On the stochastic selection of integral curves of a rough vector field} 
\author [J. Pitcho]{Jules Pitcho}
\address{Jules Pitcho
	\hfill\break  ENS  de Lyon, UMPA, 46 all\'ee d'Italie, 
	69364 Lyon,
	France}
\email{jules.pitcho@ens-lyon.fr}

\begin{document}
	\maketitle 
	\begin{abstract}
		We prove that for bounded, divergence-free vector fields $\bb$ in $L^1_{loc}((0,1];BV(\T^d;\R^d))$, there exists a unique incompressible measure on integral curves of $\bb$. We recall the vector field constructed by Depauw in \cite{Depauw}, which lies in the above class, and prove that for this vector field, the unique incompressible measure on integral curves exhibits stochasticity. 
	\end{abstract}
		\section{Introduction} 
Consider a bounded, divergence-free vector field $\bb:[0,1]\times\T^d\to \R^d$. A general principle due to Ambrosio (see Theorem \ref{thm_superposition}) guarantees existence of a measure concentrated on integral curves of $\bb$, whose 1-marginals at every time is the Lebesgue measure on $\T^d$. 
\begin{question}
Is there a robust selection criterion amongst all such measures?
\end{question}
 When $\bb$ lies in $L^1((0,1);BV(\T^d;\R^d))$, Ambrosio \cite{AmbBV} following on DiPerna and Lions \cite{DPL89} proved that there exists a unique such measure, thereby answering the above question affirmatively. In this work we give an affirmative answer to the above question for a class of bounded, divergence-free vector fields to which the work of Ambrosio does not apply. 
 Let us introduce the main objects of study. 
 
 \bigskip 
 
All vector fields $\bb:[0,1]\times\T^d\to\R^d$ will be Borel, essentially bounded and divergence-free, meaning that
$$\div_x\bb =0 \qquad\text{in the sense of distributions on \;} [0,1]\times\T^d.$$
 Let us define integral curves of $\bb$. 
		\begin{definition} 
	We shall say that a curve $\gamma:[0,1]\to \T^d$ is an integral curve of $\bb$, if it is an absolutely continuous solution to the ODE $$\dot{\gamma}(t)=\bb(t,\gamma(t)),$$
	which means explicitely that for every $s,t\in[0,1]$
	\begin{equation*}
	\gamma(t)-\gamma(s)=\int_s^t\bb(\tau,\gamma(\tau))d\tau \qquad\text{and} \qquad \int_0^1|\bb(\tau,\gamma(\tau))|d\tau<+\infty. 
	\end{equation*}
	We shall further say that $\gamma$ is an integral curve starting from $x$ at time $s$, if $\gamma(s)=x$ and we shall write $\gamma_{s,x}.$
\end{definition} 
Let $\Gamma$ denote the metric space $C([0,1];\T^d)$ of continuous paths endowed with the uniform metric, and let $\mathscr{M}$ denote the corresponding Borel $\s$-algebra. We recall that a selection of integral curves $\{\gamma_{s,x}:x\in\T^d\}$ is measurable, if the map $\T^d\ni x\longmapsto \gamma_{s,x}\in\Gamma$ is Borel. 
Ambrosio proposed the following definition in \cite{AmbBV}.
\begin{definition} 
A measurable selection $\{\gamma_{s,x}:x\in\T^d\}$ of integral curves of $\bb$ is said to be a \emph{regular measurable selection}, if 
there exists $C>0$ such that
\begin{equation}\label{eq_regular} 
\Big|\int_{\T^d}\phi(\gamma_{s,x}(t))dx\Big|\leq C \int_{\T^d}|\phi(y)|dy\qquad\forall \phi\in C_c(\T^d),\;\forall t\in[0,1].
\end{equation}
\end{definition} 
When $\bb$ belongs to $ L^1((0,1);BV(\T^d;\R^d))$, and satisfies $\int_0^1\|[\div_x\bb(s,\cdot)]_-\|_{L^\infty_x}ds<+\infty$, Ambrosio proved the existence of a regular measurable selection and the following essential uniqueness result: any two regular measurable selection $\{\gamma^1_{s,x}\}$ and $\{\gamma^2_{s,x}\}$\footnote{From now on, when we omit to specify the indexing set for a family of paths or a family of measures, it will implicitely be understood that the family is indexed by $\T^d$.} of integral curves of $\bb$
must coincide up to a set of vanishing Lebesgue measure. Moreover, if $\bb$ is divergence-free, essential uniqueness holds amongst measure-preserving measurable selection, that is those which satisfy
\begin{equation}
\Big|\int_{\T^d}\phi(\gamma_{s,x}(t))dx\Big|=\int_{\T^d}|\phi(y)|dy\qquad\forall \phi\in C_c(\T^d),\;\forall t\in[0,1].
\end{equation}

\bigskip 
Recently, Pappalettera constructed in \cite{Pappa23} a divergence-free vector field  $\bb$ for which there does not exist a measure-preserving selection of characteristics. This vector field does not belong to $L^1((0,1);BV(\T^d;\R^d))$. This, however does not exclude that for $\Leb{d}$-a.e. $x\in\T^d$, a probability measure concentrated on integral curves of $\bb$ starting from $x$ at time $0$ can be uniquely selected by some appropriate criterion. The present work investigates such a selection criterion for divergence-free vector fields in $L_{loc}^1((0,1];BV(\T^d;\R^d))$. For this class of vector fields, a selection criterion for solutions of the continuity equation using regularisation by convolution was already proved by the author in \cite{PitchoArmk23}. We here give a Lagrangian counterpart to this previous result. We begin by defining the measures on integral curves we shall study. 

\subsection{Lagrangian representations}
All measures are Radon measures in this work. We will denote by
 $e_t:\Gamma \ni \gamma \longmapsto\gamma(t)\in \T^d$ the evaluation map. We will also denote by $\rho:[0,1]\times\T^d\to \R^+$ a bounded density. It will always be assumed in $C([0,1];w^*-L^\infty(\T^d))$, which we may do without loss of generality, when the vector field $\rho(1,\bb)$ solves 
 \begin{equation}\label{eq_pde}\tag{PDE}
 \div_{t,x}\rho(1,\bb)=0\qquad \text{in the sense of distributions on $[0,1]\times\T^d$}. 
 \end{equation}
 A proof of this fact is given in \cite{PitchoArmk23}. Let us now define the main object under consideration in this paper. 
\begin{definition}\label{defn_lagrang_rep}

	We shall say that a bounded, positive measure $\eeta$ on $\Gamma$ is a Lagrangian representation of the vector field $\rho(1,\bb)$, if the following conditions hold: 
	\begin{enumerate}
		\item $\eeta$ is concentrated on the set $\Gamma$ of integral curves of $\bb$, 
		which explicitely means that for every $s,t\in [0,1]$
		\begin{equation*}
		\int_\Gamma\Big|\gamma(t)-\gamma(s)-\int_s^t\bb(\tau,\gamma(\tau))d\tau\Big|\eeta(d\gamma)=0;
		\end{equation*}
		\item for every $t\in[0,1]$, we have
		\begin{equation}
		\rho(t,\cdot)\Leb{d}=(e_t)_\#\eeta. 
		\end{equation}
	\end{enumerate}
\end{definition}

We now record the following general existence theorem for Lagrangian representations. 
\begin{theorem}[Ambrosio's superposition principle]\label{thm_superposition}
In the context of this paragraph, and assuming that $\rho(1,\bb)$ solves \eqref{eq_pde}, there exists a Lagrangian representation of $\rho(1,\bb)$. 
\end{theorem}
The above theorem is proved (see for instance \cite{ambrosiocrippaedi}) by a regularisation and compactness argument, where the hypothesis that $\rho$ is non-negative plays an essential role. 
Let us now introduce the set of Lipschitz paths with constant $L>0$ 
\begin{equation}\label{eqn_lipschitz_path}
\Gamma_L:=\Big\{\gamma\in \Gamma\;:\;|\gamma(s)-\gamma(t)|\leq L|s-t| \quad\forall s,t\in [0,1]\Big\}.
\end{equation} 
\begin{remark}\label{rmk_compact_seprable}
	$\Gamma_L$ is a compact, separable metric space with the metric induced from $\Gamma$. 
\end{remark}
We now have the following lemma.
\begin{lemma}\label{lem_conc_lipschitz}
Any Lagrangian representation $\eeta$ of $(1,\bb)$ is concentrated on $\Gamma_{\|\bb\|_{L^\infty_{t,x}}}$. 
\end{lemma}
\begin{proof}
	Let $\eeta$ be a Lagrangian representation of $(1,\bb)$. 
	Let $D$ be a countable dense subset of $[0,1]$. Let $\phi\in C_c^\infty((0,1)\times (0,1)^d)$ be a standard mollifier. Define $$\phi^\e(t,x):=\frac{1}{{\e^{d+1}}}\phi(\frac{t}{\e},\frac{x}{\e}),$$ and denote $\bb^\e=\bb\ast \phi^\e$.  Let $s,t\in D$. 
Notice that for every $\e>0$, we have
\begin{equation}
\begin{split} 
&\int_\Gamma\Big|\gamma(t)-\gamma(s)-\int_s^t\bb^\e(\tau,\gamma(\tau))d\tau\Big|\eeta(d\gamma)\\
&\leq \int_\Gamma\Big|\gamma(t)-\gamma(s)\Big|\eeta(d\gamma)+\int_\Gamma\int_s^t|\bb^\e(\tau,\gamma(\tau))d\tau|\eeta(d\gamma)\\
&\leq\int_\Gamma\int_s^t|\bb(\tau,\gamma(\tau))|d\tau\eeta(d\gamma)+\int_\Gamma\int_s^t|\bb^\e(\tau,\gamma(\tau))|d\tau\eeta(d\gamma)\\
&=\int_s^t\int_{\T^d} |\bb(\tau,x)|dxd\tau + \int_s^t\int_{\T^d} |\bb^\e(\tau,x)|dxd\tau \\
&\leq 2|t-s|\|\bb\|_{L^\infty_{t,x}},
\end{split} 
\end{equation}
where in the second to last line, we have used Fubini, as well as $(e_\tau)_\#\eeta = \Leb{d}$, and
 in the last line we have used that $\|\bb^\e\|_{L^\infty_{t,x}}\leq \|\bb\|_{L^\infty_{t,x}}$ and H\"older's inequality. 
Therefore, by the dominated convergence theorem, it holds that
\begin{equation}
\int_\Gamma\Big|\gamma(t)-\gamma(s)-\lim_{\e\downarrow 0}\int_s^t\bb^\e(\tau,\gamma(\tau))d\tau\Big|\eeta(d\gamma)=0,
\end{equation}
which implies that there exists a set $N_{s,t}\subset\Gamma$ of vanishing $\eeta$ measure such that for every $\gamma\in\Gamma-N_{s,t}$, we have
\begin{equation}
\gamma(t)-\gamma(s)=\lim_{\e\downarrow 0}\int_s^t\bb^\e(\tau,\gamma(\tau))d\tau.
\end{equation} 
Since $\|\bb^\e\|_{L^\infty_{t,x}}\leq \|\bb\|_{L^\infty_{t,x}}$, it therefore holds that for every $\gamma\in\Gamma-N_{s,t}$
\begin{equation}
|\gamma(t)-\gamma(s)|\leq |t-s|\|\bb\|_{L^\infty_{t,x}}. 
\end{equation}
Define $$N:=\bigcup_{s,t\in D} N_{s,t},$$
which is a set of vanishing $\eeta$ measure since $D$ is countable. As $s,t$ were arbitrary in $D$, and by density of $D$ in $[0,1]$, we therefore have that for every $\gamma\in \Gamma-N$, it holds 
\begin{equation*}
|\gamma(t)-\gamma(s)|\leq |t-s|\|\bb\|_{L^\infty_{t,x}} \qquad\forall s,t\in[0,1]. 
\end{equation*}
This proves the thesis. 
\end{proof}
Next we present our main tool: the disintegration of a measure with respect to a Borel map and a target measure. In the study of weak solutions of the continuity equation, disintegration has previously been used in \cite{ABC14} by Alberti, Bianchini and Crippa to establish the optimal uniqueness result for the continuity equation along a bounded, divergence-free, and autonomous in the two-dimensional setting. In \cite{BianchiniBonicatto20}, Bianchini and Bonicatto also used disintegration to prove a uniqueness result for nearly incompressible vector fields in $L^1_tBV_x$. In view of Lemma \ref{lem_conc_lipschitz}, we will identify a Lagrangian representation of $(1,\bb)$ with its restriction to the Borel $\s$-algebra of the compact set $\Gamma_{\|\bb\|_{L^\infty_{t,x}}}$, and thanks to Remark \ref{rmk_compact_seprable}, we will be able to perform a disintegration of this measure. 
	\subsection{Disintegration of measures} \label{sec_disintegration}
Let $X$ and $Y$ be compact, separable metric spaces, $\mmu$ a measure on $X$, $f:X\to Y$ a Borel map, $\nnu$ a measure on $Y$ such that $f_\#\mmu \ll \nnu$. Then there exists a Borel family $\{\mmu_y : y\in Y\}$ of measures on $X$ such that 
\begin{enumerate}
	\item $\mmu_y$ is supported on the level set $E_y:=f^{-1}(y)$ for every $y\in Y$;
	\item the measure $\mmu$ can be decomposed as $\mmu=\int_Y\mmu_yd\nnu(y)$, which means that 
	\begin{equation}\label{eqn_disintegration}
	\mmu(A)=\int_Y\mmu_y(A)d\nnu(y) 
	\end{equation}
	for every Borel set $A$ contained in $X$. 	
\end{enumerate}
If we further assume that $\mmu$ and $\nnu$ are positive measures, and that $f_\#\mmu=\nnu$, then there exists a Borel family $\{\mmu_y : y\in Y\}$ of \emph{probability} measures on $X$ satisfying $(i)$ and $(ii)$.

Any family satisfying $(i)$ and $(ii)$ is called a \emph{disintegration} of $\mmu$ with respect to $f$ and $\nnu$. The disintegration is essentially unique in the following sense: for any other disintegration $\{\tilde{\mmu}_y:y\in Y\}$ there holds $\mmu_y=\tilde\mmu_y$ for $\nnu$-a.e. $y\in Y$. 
The above facts are cited from \cite{ABC14}, and are proven in Dellacherie and Meyer \cite{DellacherieMeyer}.

\bigskip 
We now give a useful fact.
Let $g:X\to X$ a Borel map 
 such that:
\begin{itemize}
\item[(P)] for every $y\in Y$, we have $g^{-1}((f^{-1}(y))^c)=(f^{-1}(y))^c$. 
 \end{itemize} 
The following is true. 
\begin{lemma} \label{lem_disintegration_push_forward} 
	In the context of this paragraph, if $\{\mmu_y : y\in Y\}$ is a disintegration of $\mmu$ with respect to $ f$ and $\nnu$, then $\{g_\# \mmu_y : y\in Y\}$ is a disintegration of $g_\#\mmu$ with respect to $ f $ and $\nnu$. 
\end{lemma} 
\begin{proof}
	Let $y\in Y$. Observe that $g_\#\mmu_y((f^{-1}(y))^c)=\mu_y(g^{-1}(f^{-1}(y))^c)=\mu_y((f^{-1}(y))^c)=0$ where we have used (P) in the second to last equality and that $\mmu_y$ is concentrated on $f^{-1}(y)$ in the last equality. 
So	$g_\#\mmu_y$ is supported on $f^{-1}(y)$, and since $y$ was arbitrary, this proves $(i)$. 
	
	Let $A$ a Borel set in $X$. Then, as $g^{-1}(A)$ is a Borel set in $X$, it follows that 
	\begin{equation}
	g_\#\mmu(A)=\mmu(g^{-1}(A))=\int_Y \mmu_y (g^{-1}(A))d\nnu(y)=\int_Y g_\#\mmu_y (A)d\nnu(y),
	\end{equation}
	which gives $(ii)$. 
\end{proof}

 We also have the following property of the disintegration, which we will use in this paper: 
	\begin{equation}\label{eqn_int_disintegration}
	\int_X \phi d\mmu=\int_Y\Big[\int_{E_y}\phi d\mmu_y\Big]d\nnu(y),
	\end{equation}
	for every Borel function $\phi:X\to [0,+\infty]$. 
	
\subsection{Uniqueness of regular measurable selection} Ambrosio \cite{AmbBV} proved the existence and essential uniqueness of regular measurable selections in the bounded variation setting, thereby extending the work of DiPerna and Lions \cite{DPL89}. We recall that $\rho:[0,1]\times \T^d\to\R^+$ is assumed to be in $C([0,1];w^*-L^\infty(\T^d))$. The following can be extracted from Ambrosio \cite{AmbBV}. 
\begin{theorem}\label{thm_unique_lagrang_rep}
Assume that $\bb$ belongs to $L^1((0,1);BV(\T^d;\R^d))$ and that  $\rho(1,\bb)$ solves \eqref{eq_pde}. Then, there exists a unique Lagrangian representation $\eeta$ of $\rho(1,\bb)$, which further has the following property. For every $s\in[0,1]$, there exists a regular measurable selection $\{\gamma_{s,x}\}$ of integral curves of $\bb$ such that
\begin{equation}
\eeta=\int_{\T^d}\ddelta_{\gamma_{s,x}}\rho(s,x)dx. 
\end{equation}
\end{theorem} 
The above theorem implies that, if two bounded vector fields $\rho(1,\bb)$ and $\tilde\rho(1,\bb)$ solve \eqref{eq_pde} and satisfy $\rho(s,x)=\tilde\rho(s,x)$ for $\Leb{d}$-a.e. $x\in \T^d$, then $\rho=\tilde\rho$, which is the uniqueness result of Ambrosio for the Cauchy problem for the continuity equation with a vector field in $L^1_tBV_x$. 
The essential uniqueness of regular measurable selections can then be deduced. We record it in the following remark. 
\begin{remark}\label{rmk_unique_regular_measur_selec}
	Under the hypothesis of Theorem \ref{thm_unique_lagrang_rep}, if $\{\gamma^1_{s,x}\}$ and $\{\gamma^2_{s,x}\}$ are two regular measurable selection of integral curves of $\bb$, then for $\Leb{d}$-a.e. $x\in \T^d$, we have $\gamma^1_{s,x}=\gamma^2_{s,x}$. Indeed, let $\bar\rho \in L^\infty(\T^d)$ with $\bar\rho\geq 0$, and define the measures
	\begin{equation*}
	\eeta^i=\int_{\T^d}\ddelta_{\gamma^i_{s,x}}\bar\rho(x)dx\qquad \text{for \;}i=1,2.
	\end{equation*}
Then, consider the densities $\rho^1,\rho^2:[0,1]\times\T^d\to \R^+$, which lie in $C([0,1];w^*-L^\infty(\T^d))$ and are given by 
\begin{equation*}
\rho^i(t,\cdot)\Leb{d}=(e_t)_\#\eeta^i\qquad \text{for \;} i=1,2.
\end{equation*}
The vector fields $\rho^i(1,\bb)$ both solve \eqref{eq_pde}. Therefore, by Theorem \ref{thm_unique_lagrang_rep}, we have 
\begin{equation*}
\int_{\T^d}\ddelta_{\gamma^1_{s,x}}\bar\rho(x)dx=\int_{\T^d}\ddelta_{\gamma_{s,x}}\bar\rho(x)dx=\int_{\T^d}\ddelta_{\gamma^2_{s,x}}\bar\rho(x)dx.
\end{equation*}
As $\bar\rho$ was an arbitrary bounded, nonnegative function, and as the $\s$-algebra $\mathscr{M}$ of $\Gamma$ is countably generated, we have for $\Leb{d}$-a.e. $x\in\T^d$
\begin{equation*}
\ddelta_{\gamma_{s,x}^1}=\ddelta_{\gamma_{s,x}}=\ddelta_{\gamma^2_{s,x}},
\end{equation*}
which implies that for $\Leb{d}$-a.e. $x\in\T^d$
\begin{equation*}
\gamma_{s,x}^1=\gamma_{s,x}=\gamma^2_{s,x}.
\end{equation*}
\end{remark} 

Given $\tau>0,$ we define the truncated versions of $\bb$
\begin{equation}\label{eqn_b_truncated} 
\bb^\tau(t,x):=\left\{ 
\begin{split} 
\bb(t,x)\qquad &\text{if} \quad t\geq \tau,\\
0\qquad &\text{if}\quad t<\tau.
\end{split}
\right. 
\end{equation}
Under the assumption that the bounded variation norm of $\bb$ is not integrable at time zero, the following essential uniqueness of regular measurable selections of integral curves still holds.
	\begin{proposition} \label{prop_unique_regular_measurable}
Let $s\in(0,1]$.	Assume that $\bb$ belongs to $L_{loc}^1((0,1];BV(\T^d;\R^d)),$ and consider two regular measurable selections $\{\gamma^1_{s,x}\}$ and $\{\gamma^2_{s,x}\}$ of integral curves of $\bb$ starting from $s$. Then $\gamma^1_{s,x}=\gamma^2_{s,x}$ for $\Leb{d}$-a.e. $x\in\T^d.$  
\end{proposition}
\begin{proof}
Let $k\in\N$.  It can be verified directly that the two measurable selections $\{\gamma^1_{s,x}(1/k\vee \cdot)\}$ and $\{\gamma^2_{s,x}(1/k\vee \cdot)\}$ are regular measurable selections of integral curves of the vector field $\bb^{1/k}$ defined in \eqref{eqn_b_truncated}. Note that this vector field belongs to $L^1((0,1);BV(\T^d;\R^d))$, so in view of Remark \ref{rmk_unique_regular_measur_selec}, we have $\gamma_{s,x}^1(1/k\vee\cdot)=\gamma^2_{s,x}(1/k\vee \cdot)$ for every $x\in \T^d-N_k$, where $N_k$ is a set of vanishing Lebesgue measure. 
Define $$N:=\bigcup_{k\in\N} N_k,$$
which is of vanishing Lebesgue measure. 
Then, for every $k\in\N$, we have $\gamma_{s,x}^1(1/k\vee\cdot)=\gamma^2_{s,x}(1/k\vee \cdot)$ for every $x\in \T^d-N,$ which implies $\gamma_{s,x}^1=\gamma^2_{s,x}$ for every $x\in \T^d-N$ by continuity. The thesis follows.
\end{proof}


\subsection{Statement of results} 
In this paper, the vector field $\bb$ will satisfy the hypothesis of Proposition \ref{prop_unique_regular_measurable}. Accordingly, we fix for the rest of the paper
$\{\gamma_{1,y}\}$ a (essentially unique) regular measurable selection of integral curves of $\bb$ starting from time $1$. 
$\bb_{DP}$ denotes the bounded, divergence-free vector field in $L_{loc}^1((0,1];BV(\T^2;\R^2))$ constructed by Depauw in \cite{Depauw}. For completeness we give a construction of $\bb_{DP}$ in the Appendix. 
We now state our main theorem. 
\begin{theorem}\label{thm_main}
	Consider a bounded, divergence-free vector field $\bb:[0,1]\times\T^d\to\R^d$. 
Assume that $\bb\in L_{loc}^1((0,1];BV(\T^d;\R^d))$.	Then, there exists a unique Lagrangian representation $\eeta$ of $(1,\bb)$, which furthermore has the following properties:
	\begin{enumerate}
	\item  the family of probability measures $\{\ddelta_{\gamma_{1,y}}\}$ is a disintegration of $\eeta$ with respect to $e_1$ and $\Leb{d}$;

	\item for every smooth sequence $(\bb^k)_{k\in\N}$ such that $\bb^k\rightarrow \bb$ in $L^1_{loc}$, $\|\bb^k\|_{L^\infty_{t,x}}\leq \|\bb\|_{L^\infty_{t,x}}$, and $\div_x\bb^k=0$, the unique Lagrangian representation of $(1,\bb^k)$ converges narrowly to $\eeta$ as $k\to+\infty$;
	
	\item there exists a Borel family of probability measures $\{\tilde\nnu_x\}$ on $\T^d$ such that any disintegration $\{\eeta_{0,x}\}$ of $\eeta$ with respect to $e_0$ and $\Leb{d}$ satisfies
	\begin{equation}
	\eeta_{0,x}=\int_{\T^d}\ddelta_{\gamma_{1,y}}\tilde\nnu_x(dy),
	\end{equation}
	for $\Leb{d}$-a.e. $x\in\T^d$. 
	
	Moreover, for the vector field $\bb_{DP}$, the measure $\tilde\nnu_x$ is not a Dirac mass for $\Leb{2}$-a.e. $x\in\T^2$. 
	\end{enumerate} 
\end{theorem}
Observe that a sequence satisfying the hypothesis of $(ii)$ can be generated by regularising $\bb$ by convolution. 
We further ask the following question on the measures $\tilde\nnu_x$ of the above theorem, which capture the stochastic behaviour of the Lagrangian representation $\eeta$. 
\begin{question}
	Can one construct a bounded, divergence-free vector field in $L^1_{loc}((0,1];BV(\T^d;\R^d))$ such that the measures $\tilde\nnu_x$ are absolutely continuous with respect to $\Leb{d}$ for $\Leb{d}$-a.e. $x\in\T^d$?
\end{question}
We finally note that the class of vector field under study in this article has been previously investigated in \cite{AmbCripFigSpin09} and that stochastic selection has been investigated for a toy model in \cite{MailybaevCat}. 

\subsection{Plan of the paper}
In Section \ref{sec_uniqueness}, we prove that there exists a unique Lagrangian representation of $(1,\bb)$ under the hypothesis of Theorem \ref{thm_main}, as well as $(i)$ and $(ii)$ of Theorem \ref{thm_main}. In Section \ref{sec_stoch}, we prove $(iii)$ of Theorem \ref{thm_main}. In the Appendix, we give for completeness a construction of $\bb_{DP}$, the vector field constructed by Depauw in \cite{Depauw}. 

\bigskip 
\subsection*{Acknowledgements} The author is thankful to his advisor Nikolay Tzvetkov for his support. The author is thankful to Stefano Bianchini for enlightening discussions on measure theory and on the disintegration of a measure. The author acknowledges the hospitality of the Pitcho Centre for Scientific Studies where this work was done. 

\section{Uniqueness of the Lagrangian representation} \label{sec_uniqueness}
In this section $\bb:[0,1]\times\T^d\to\R^d$ is an essentially bounded, divergence-free Borel vector field satisfying the assumptions of Theorem \ref{thm_main}, namely $\bb\in L^1_{loc}((0,1];BV(\T^d;\R^d))$. 
\subsection{Backwards stopping of the Lagrangian representation} 
 Let $\eeta$ be a Lagrangian representation of $(1,\bb)$, which exists by Theorem \ref{thm_superposition}, and let $ \{\eeta_{1,y}\}$ be a disintegration with respect to $e_1$ and $\Leb{d}$. Let $\{\gamma_{1,y}\}$ be a (essentially unique) regular measurable selection of integral curves of $\bb$. Our goal is now to prove that for $\Leb{d}$-a.e. $y\in\T^d$, we have
 \begin{equation}\label{eqn_eeta_1,x}
 \eeta_{1,y}=\ddelta_{\gamma_{1,y}},
 \end{equation} 
 which will imply by definition of the disintegration that $$\eeta=\int_{\T^d}\ddelta_{\gamma_{1,y}}dy,$$
 from which uniqueness of the Lagrangian representation of $(1,\bb),$ as well as part $(i)$ of Theorem \ref{thm_main} will follow. Given two positive real numbers $a$ and $b$, we will write $a\vee b =\max\{a,b\}.$
  Let $\tau>0$ and
consider the backward stopping map $S^\tau :\Gamma \ni \gamma(\cdot) \longmapsto \gamma( \tau \vee\cdot)\in\Gamma$. Note that $S^\tau$ clearly satisfies (P) of Section \ref{sec_disintegration} with $X=\Gamma_{\|\bb\|_{L^\infty_{t,x}}}$, $Y=\T^d$, $f=e_1$, and $g=S^\tau$. Therefore, by Lemma \ref{lem_disintegration_push_forward}, $ \{(S^\tau)_\#\eeta_{1,y}\}$ is a disintegration with respect to $e_1$ and $\Leb{d}$. 
 For simplicity, we will write $\eeta^\tau:=(S^\tau)_\#\eeta$ and $\eeta_{1,y}^\tau:=(S^\tau)_\#\eeta_{1,y}$. 
Recall that we have defined in \eqref{eqn_b_truncated} the truncated version $\bb^\tau$ of $\bb$.
We then have the following lemma. 
\begin{lemma}\label{lem_form_of_disint}
	For every $\tau>0$, the family  $\{\ddelta_{\gamma_{1,y}(\tau\vee\cdot)}\}$ is a disintegration of $\eeta^\tau$ with respect to $e_1$ and $\Leb{d}$. 
\end{lemma}
\begin{proof} 
As $\bb^\tau$ belongs to $L^1((0,1);BV(\T^d;\R^d))$ is bounded, and divergence-free, there is an essentially unique regular measurable selection $\{\gamma^\tau_{1,y}\}$ of integral curves of $\bb^\tau$ thanks to Remark \ref{rmk_unique_regular_measur_selec}.
Now, observe that $\{\gamma_{1,y}(\tau\vee \cdot)\}$ is a regular measurable selection of integral curves of $\bb^\tau$, hence for $\Leb{d}$-a.e. $y\in\T^d$, we have that $\gamma^\tau_{1,y}( \cdot)=\gamma_{1,y}(\tau\vee\cdot)$. 
It can be checked directly that $\eeta^\tau$ is a Lagrangian representation of $(1,\bb^\tau)$. So by Theorem \ref{thm_unique_lagrang_rep}, $\{\ddelta_{\gamma^\tau_{1,y}}\}$ is a disintegration of $\eeta^\tau$ with respect to $e_1$ and $\Leb{d}$ so that 
\begin{equation*}
\eeta^\tau=\int_{\T^d}\ddelta_{\gamma^\tau_{1,y}}dy.
\end{equation*}
Therefore, by essential uniqueness of the disintegration, $\{\ddelta_{\gamma_{1,y}(\tau\vee \cdot)}\}$ is a disintegration of $\eeta^\tau$ with respect to $e_1$ and $\Leb{d}$  so that 
\begin{equation*}
\eeta^\tau=\int_{\T^d}\ddelta_{\gamma_{1,y}(\tau\vee \cdot)}dy.
\end{equation*}
\end{proof}
We also have the following simple fact.

\begin{lemma}\label{lem_conv_lagrang_rep} 
	Let $\mmu$ be a probability measure on $\Gamma$.
Then $(S^\tau)_\#\mmu$ converges narrowly to $\mmu$ as $\tau\downarrow 0$. 
\end{lemma}

\begin{proof}
	Let $\Phi \in C_b(\Gamma)$. For every $\gamma\in \Gamma$, we have $\lim_{\tau\downarrow0} S^\tau\gamma=\gamma$. Then we have $\lim_{\tau\downarrow 0}\Phi(S^\tau\gamma)=\Phi(\gamma)$ by continuity of $\Phi$. Also, we clearly have 
	\begin{equation*} 
	\int_\Gamma| \Phi(S^\tau\gamma)|\mmu(d\gamma)\leq \|\Phi\|_{C^0} \qquad\forall \tau>0.
	\end{equation*} 
	 So, by dominated convergence, it holds that 
	\begin{equation*}
	\lim_{\tau\downarrow 0} \int_\Gamma \Phi(\gamma)(S^\tau)_\#\mmu(d\gamma)=	\lim_{\tau\downarrow 0} \int_\Gamma \Phi(S^\tau\gamma)\mmu(d\gamma)=\int_\Gamma \Phi(\gamma)\mmu(d\gamma).
	\end{equation*}
	Since $\Phi$ was arbitrary in $C_b(\Gamma)$, the thesis follows. 
\end{proof}

\subsection{Proof of $(i)$ of Theorem \ref{thm_main}}
\begin{proof}
Recall that $\eeta$ is a Lagrangian representation of $(1,\bb)$ and that $\{\eeta_{1,y}\}$ is a disintegration of $\eeta$ with respect to $e_1$ and $\Leb{d}$. Recall also that $\{\gamma_{1,y}\}$ is a (essentially unique) regular measurable selection of integral curves of $\bb$. Let us prove \eqref{eqn_eeta_1,x}, which will yield both uniqueness of the Lagrangian representation of $(1,\bb)$ as well as $(i)$ of Theorem \ref{thm_main}, namely it will show
	\begin{equation}\label{eqn_eeta_disnt}
	\eeta=\int_{\T^d}\ddelta_{\gamma_{1,y}}dy.
	\end{equation}
	By separability of $C_c(\Gamma)$, there exists a countable subset $\mathcal{N}$ of $C_c(\Gamma)$, which is dense. 
	Let $\tau>0,$ and let $\{\eeta_{1,y}\}$ be a disintegration of $\eeta$ with respect to $e_1$ and $\Leb{d}$. By Lemma \ref{lem_disintegration_push_forward}, $\{\eeta_{1,y}^\tau\}$ is a disintegration of $\eeta^\tau$ with respect to $e_1$ and $\Leb{d}$. By Lemma \ref{lem_form_of_disint}, and by essential uniqueness of the disintegration, we have $\ddelta_{\gamma_{1,y}(\tau\vee\cdot)}=\eeta_{1,y}^\tau$ for $\Leb{d}$-a.e. $y\in \T^d$. Let $\Phi\in \mathcal{N}$, and let $B$ be a Borel set in $\T^d$. We then have that
\begin{equation}
\begin{split} 
\int_B \int_\Gamma \Phi(\gamma)\ddelta_{\gamma_{1,y}}(d\gamma)dy&=\int_{B}\lim_{\tau\downarrow 0}\int_\Gamma \Phi(\gamma)\ddelta_{\gamma_{1,y}(\tau\vee \cdot)}(d\gamma)dy,\\
&=\int_{B}\lim_{\tau\downarrow 0}\int_\Gamma \Phi(\gamma)\eeta^{\tau}_{1,y}(d\gamma)dy,\\
&=\int_{B}\int_\Gamma\Phi(\gamma)\eeta_{1,y}(d\gamma)dy,
\end{split} 
\end{equation}
where in the first equality, we have used that $\ddelta_{\gamma_{1,y}(\tau\vee \cdot)}$ converges narrowly to $\ddelta_{\gamma_{1,y}}$ as $\tau\downarrow 0$ by Lemma \ref{lem_conv_lagrang_rep}. In the second to last equality, we have used that $\ddelta_{\gamma_{1,y}(\tau\vee\cdot)}=\eeta_{1,y}^\tau$ for $\Leb{d}$-a.e. $y\in \T^d$, which follows from Lemma \ref{lem_form_of_disint}. In the last equality, we have used Lemma \ref{lem_conv_lagrang_rep}. 
As $B$ was an arbitrary Borel set of $\T^d$, we have that there exists a set $N_\Phi$ of vanishing Lebesgue measure such that for every $y\in \T^d-N_\Phi$, we have
$$\int_\Gamma \Phi(\gamma)\ddelta_{\gamma_{1,y}}(d\gamma)=\int_\Gamma \Phi(\gamma)\eeta_{1,y}(d\gamma).$$
Now, let $$N:=\bigcup_{\Phi\in \mathcal{N}}N_\Phi.$$
It is a set of vanishing Lebesgue measure as $\mathcal{N}$ is countable, and for every $y\in \T^d-N$, we have
$$\int_\Gamma \Phi(\gamma)\ddelta_{\gamma_{1,y}}(d\gamma)=\int_\Gamma \Phi(\gamma)\eeta_{1,y}(d\gamma) \qquad\forall \Phi\in \mathcal{N}.$$
By density of $\mathcal{N}$ in $C_c(\Gamma)$, 
for every $y\in \T^d-N$, we have
$$\int_\Gamma \Phi(\gamma)\ddelta_{\gamma_{1,y}}(d\gamma)=\int_\Gamma \Phi(\gamma)\eeta_{1,y}(d\gamma) \qquad\forall \Phi\in C_c(\Gamma).$$
This proves that for $\Leb{d}$-a.e. $y\in \T^d$, we have $\eeta_{1,y}=\ddelta_{\gamma_{1,y}}$, which proves \eqref{eqn_eeta_disnt}. As $\eeta$ was an arbitrary Lagrangian representation of $(1,\bb)$, this proves both that there exists a unique Lagrangian representation of $(1,\bb)$, and $(i)$ of Theorem \ref{thm_main}.
\end{proof}

We will now prove that the unique Lagrangian representation of $(1,\bb)$ can be obtained as the unique limit of Lagrangian representations of suitable regularisations of $(1,\bb)$. 

\subsection{Proof of $(ii)$ of Theorem \ref{thm_main}}
\begin{proof} 
Let $(\bb^k)_{k\in\N}$ be a smooth sequence such that $\bb^k\to\bb$ in $L^1_{loc}$ and such that, for every $k\in\N$, we have
\begin{equation}
\sup_{(t,x)\in [0,1]\times \T^d} |\bb^k(t,x)|\leq \|\bb\|_{L^\infty_{t,x}},
\end{equation} and such that $\div_x\bb^k=0$. 
Let $\XX^k:[0,1]\times \T^d\to\T^d$ be the unique flow along $\bb^k$, namely $\XX^k$ solves
\begin{equation}
\left\{ 
\begin{split} 
\partial_t \XX^k(t,x)&=\bb^k(\XX^k(t,x)),\\
\XX^k(0,x)&=x.
\end{split} 
\right. 
\end{equation}

The measure defined by $$\eeta^k(A)=\int_{\T^d}\ddelta_{\XX^k(\cdot,x)}(A)dx,$$ 
for every Borel set $A$ in $\Gamma$ is then the unique Lagrangian representation of $(1,\bb^k)$, as we clearly have $(e_t)_\#\eeta^k=\XX^k(t,\cdot)_\#\Leb{d}=\Leb{d},$ since $\bb^k$ is divergence-free, and also that $\eeta^k$ is clearly concentrated on integral curves of $\bb^k$. 

\bigskip 

\textbf{Step 1.} Compactness. Recall the definition of the space $\Gamma_L$ of Lipchitz paths with Lipschitz constant $L>0$ given in \eqref{eqn_lipschitz_path}. 
In view of Lemma \ref{lem_conc_lipschitz}, we have that 
$\eeta^k(\Gamma_{\|\bb\|_{L^\infty_{t,x}}})=1$ for every $k\in\N$.
As $\Gamma_{\|\bb\|_{L^\infty_{t,x}}}$ is compact by Remark \ref{rmk_compact_seprable}, it follows by Prokhorov theorem, that there exists an increasing map $\xi :\N\to\N$ such that 
	$\eeta^{\xi(k)}$ converges narrowly to some probability measure $\eeta$ on $\Gamma_{\|\bb\|_{L^\infty_{t,x}}}$ as $k\to+\infty.$

\bigskip 

\textbf{Step 2.} Let us prove that $\eeta$ is a Lagrangian representation of $(1,\bb)$. Let $\phi\in C(\T^d)$ and $t\in [0,1]$. We have that $\Gamma\ni \gamma \longmapsto \phi(e_t(\gamma))$ is in $C_b(\Gamma)$. Therefore,
\begin{equation*}
\int_\Gamma \phi(e_t(\gamma))\eeta(d\gamma)=\lim_{k\to+\infty}\int_\Gamma \phi(e_t(\gamma))\eeta^{\xi(k)}(d\gamma)=\int_{\T^d}\phi(x)dx.
\end{equation*}
As $\phi$ and $t$ were arbitrary, this implies that $(e_t)_\#\eeta=\Leb{d}$ for every $t\in[0,1]$. 
We still need to prove that $\eeta$ is concentrated on integral curves of $\bb.$
 Let $s,t\in[0,1]$. We have to check that
\begin{equation}\label{eqn_eeta_integral_curves}
\int_\Gamma \Big|\gamma(t)-\gamma(s)-\int_s^t\bb(\tau,\gamma(\tau))d\tau\Big|\eeta(d\gamma)=0.
\end{equation}
We know that 
\begin{equation*}
\int_\Gamma \Big|\gamma(t)-\gamma(s)-\int_s^t\bb(\tau,\gamma(\tau))d\tau\Big|\eeta^{\xi(k)}(d\gamma)=0,
\end{equation*}
however we cannot pass into the limit $k\to +\infty$ in the above equation because the functional 
\begin{equation}
\Gamma \ni \gamma \longmapsto \Big|\gamma(t)-\gamma(s)-\int_s^t\bb(\tau,\gamma(\tau))d\tau\Big|,
\end{equation}
need not be continuous since $\bb$ is not continuous. To circumvent this problem,  let $\e>0$ and let $\cc:[0,1]\times \T^d\to \R^d$ be a continuous vector field such that $\int_s^t|\cc(\tau,x)-\bb(\tau,x)|d\tau dx<\e$. We then have
\begin{equation*}
\begin{split} 
&\int_\Gamma \Big|\gamma(t)-\gamma(s)-\int_s^t\bb(\tau,\gamma(\tau))d\tau\Big|\eeta(d\gamma)\\
&\overset{1}{\leq} \int_\Gamma\Big|\gamma(t)-\gamma(s)-\int_s^t\cc(\tau,\gamma(\tau))d\tau\Big|\eeta(d\gamma)+\int_\Gamma \Big|\int_s^t\cc(\tau,\gamma(\tau))-\bb(\tau,\gamma(\tau))\Big|\eeta(d\gamma)\\
&\overset{2}{\leq} \limsup_{k\to +\infty} \int_\Gamma\Big|\gamma(t)-\gamma(s)-\int_s^t\cc(\tau,\gamma(\tau))d\tau\Big|\eeta^k(d\gamma) +\int_\Gamma \Big|\int_s^t\cc(\tau,\gamma(\tau))-\bb(\tau,\gamma(\tau))\Big|\eeta(d\gamma)\\
&\overset{3}{=} \limsup_{k\to +\infty} \int_\Gamma\Big|\int_s^t(\bb^k(\tau,\gamma(\tau))-\cc(\tau,\gamma(\tau)))d\tau\Big|\eeta^k(d\gamma) +\int_\Gamma \Big|\int_s^t(\cc(\tau,\gamma(\tau))-\bb(\tau,\gamma(\tau)))d\tau\Big|\eeta(d\gamma)\\
&\overset{4}{\leq} \limsup_{k\to +\infty} \int_\Gamma\int_s^t\Big|\bb^k(\tau,\gamma(\tau))-\cc(\tau,\gamma(\tau))\Big|d\tau\eeta^k(d\gamma) +\int_\Gamma \int_s^t\Big|\cc(\tau,\gamma(\tau))-\bb(\tau,\gamma(\tau))\Big|d\tau\eeta(d\gamma)\\
&\overset{5}{=}  \limsup_{k\to +\infty} \int_\Gamma\int_s^t\Big|\bb^k(\tau,x)-\cc(\tau,x)\Big|d\tau dx +\int_\Gamma \int_s^t\Big|\cc(\tau,x)-\bb(\tau,x)\Big|d\tau dx\\
&\overset{6}{\leq}  \limsup_{k\to +\infty} \int_{\T^d}\int_s^t\Big|\bb^k(\tau,x)-\bb(\tau,x)\Big|d\tau dx +2\int_{\T^d} \int_s^t\Big|\cc(\tau,x)-\bb(\tau,x)\Big|d\tau dx\\
&\overset{7}{=}2\int_{\T^d} \int_s^t\Big|\cc(\tau,x)-\bb(\tau,x)\Big|d\tau dx<2\e.
\end{split} 
\end{equation*}
1 follows by a triangular inequality, 2 follows because the functional 
\begin{equation*}
\Gamma \ni \gamma \longmapsto \Big|\gamma(t)-\gamma(s)-\int_s^t\cc(\tau,\gamma(\tau))d\tau\Big|
\end{equation*}
is continuous, 3 follows because $\eeta^k$ is concentrated on integral curves of $\bb^k$, 4 follows by bringing the absolute value inside the integral, 5 follows because $(e_\tau)_\#\eeta^k=\Leb{d}=(e_\tau)_\#\eeta,$ 6 follows by a triangular inequality, and 7 follows since $\bb^k\to \bb$ in $L^1_{loc}$. As $\e$ was arbitrary, \eqref{eqn_eeta_integral_curves} follows. Therefore $\eeta$ is a Lagrangian representation of $(1,\bb)$. By the first part of Theorem \ref{thm_main} we have already proved, we know that there exists a unique Lagrangian representation $\eeta$ of $(1,\bb)$. Therefore, the whole sequence $\eeta^k$ converges narrowly to $\eeta$ as $k\to+\infty$. This proves the thesis.

\end{proof}

\bigskip 

\section{Stochasticity} \label{sec_stoch}
We will now prove part $(iii)$ of Theorem \ref{thm_main}. Throughout this section, $\eeta$ is the unique Lagrangian representation of $(1,\bb)$ from the first part of Theorem \ref{thm_main}. 
Recall that we have fixed a regular measurable selection $\{\gamma_{1,y}\}$ of integral curves of $\bb$ starting from $1$ and that $$\eeta=\int_{\T^d}\ddelta_{\gamma_{1,y}}dy.$$
We also define the measure $\nnu =(e_0,e_1)_\#\eeta$ on $\T^d\times\T^d$.
 For every $x\in \T^d,$ we define the family of measures on $\Gamma$
\begin{equation}\label{eqn_def_ddelta_x_gamma}
\ddelta_{x,\gamma_{1,y}}:=\left\{ 
\begin{split} 
\ddelta_{\gamma_{1,y}}\qquad &\text{if} \quad \gamma_{1,y}(0)=x,\\
0\qquad &\text{if}\quad \gamma_{1,y}(0)\neq x.
\end{split}
\right. 
\end{equation}
Define the projection maps $$\pi_0 :\T^d\times\T^d\ni (x,y)\longmapsto x\in \T^d,$$ and $$\pi_1 :\T^d\times\T^d\ni (x,y)\longmapsto y\in \T^d.$$ Let $\{\nnu_x\}$ be a disintegration of $\nnu$ with respect to $\pi_0$ and $\Leb{d}$. 

\subsection{Disintegration of $\eeta$ with respect to $\nnu$}\label{sec_disint}
We will now give an expression for disintegrations of $\eeta$ with respect to $e_0$ and $\Leb{d}$ in terms of $\{\nnu_x\}$. Throughout this section $\{\eeta_{0,x}\}$ is a disintegration of $\eeta$ with respect to $e_0$ and $\Leb{d}$. We then define $\tilde\nnu_x :=(\pi_1)_\#\nnu_x$ for every $x\in \T^d$. We also define the probability measure 
\begin{equation} \label{eqn_def_nu_y}
\nnu_y:=\ddelta_{(\gamma_{1,y}(0),y)},
\end{equation}
 on $\T^d\times \T^d$ for every $y\in \T^d$. 
\begin{lemma} \label{lem_disint_nu}
The family	$\{\nnu_y\}$ is a disintegration of $\nnu$ with respect to $\pi_1$ and $\Leb{d}.$ 
\end{lemma}	
\begin{proof}
	It is clear that $\nnu_y$ is supported on $\pi_1^{-1}(y)$. 
	By part $(i)$ of Theorem \ref{thm_main}, we know also that $\{ \ddelta_{\gamma_{1,y}} \}$ is a disintegration of $\eeta$ with respect to $e_1$ and $\Leb{d}$. Therefore, for every Borel set $A$ in $\T^d\times \T^d$, we have $$\nnu(A)=(e_0,e_1)_\#\eeta(A)= \int_{\T^d} (e_0,e_1)_\#\ddelta_{\gamma_{1,y}} (A)dy=\int_{\T^d} \ddelta_{(\gamma_{1,y}(0),y)}(A)dy=\int_{\T^d}\nnu_y(A)dy,$$
	which proves the thesis.
\end{proof}

\begin{lemma}\label{lem_disint_eeta_nu}
The family $\{\ddelta_{x,\gamma_{1,y}}: x,y\in \T^d\}$ is a disintegration of $\eeta$ with respect to $(e_{0},e_1)$ and $\nnu$. 
	\end{lemma} 

\begin{proof}
	It is clear that $\ddelta_{x,\gamma_{1,y}}$ is supported on $(e_0,e_1)^{-1} (x,y)$. 
For every Borel set $A$ contained in $\Gamma$, we have
	\begin{equation}
	\begin{split} 
	\int_{\T^d\times\T^d}\ddelta_{x,\gamma_{1,y}}(A) \nnu(dx,dy)&=\int_{\T^d}\int_{\T^d\times \{y\}} \ddelta_{x,\gamma_{1,y}}(A) d\nnu_ydy\\
	&=\int_{\T^d} \ddelta_{\gamma_{1,y}}(A)dy\\
	&=\eeta(A),
	\end{split}
	\end{equation}
	where in the first equality we have used Lemma \ref{lem_disint_nu}, as well as \eqref{eqn_int_disintegration}. In the second equality we have used the definition \eqref{eqn_def_nu_y} of $\nnu_y$ and the definition \eqref{eqn_def_ddelta_x_gamma} of $\ddelta_{x,\gamma_{1,y}}$, and in the last equality we have used that $\{\ddelta_{\gamma_{1,y}}\}$ is a disintegration of $\eeta$ with respect to $e_1$ and $\Leb{d}$, which follows from $(i)$ of Theorem \ref{thm_main}, which we have already proved. 
	This proves the claim.
\end{proof}
\begin{lemma}\label{lem_disint_nu_x}
For $\Leb{d}$-a.e. $x\in \T^d$, we have
	\begin{equation}
	\eeta_{0,x}=\int_{\T^d}\ddelta_{\gamma_{1,y}}d\tilde\nnu_x.
	\end{equation}
\end{lemma}

	\begin{proof}
	Let $B$ be a Borel set contained in $\T^d$. As $\Gamma$ is separable, its Borel $\s$-algebra is generated by a countable family $\mathscr{G}$. Let $A\in \mathscr{G}$. We then have
	\begin{equation}
	\begin{split}
\int_B \eeta_{0,x}(A)dx
&\overset{1}{=}\int_{\T^d} \eeta_{0,x}(A\cap\{\gamma(0)\in B\})dx\\
&\overset{2}{=}\eeta(A\cap\{\gamma(0)\in B\})\\
&\overset{3}{=}\int_{\T^d\times\T^d} \ddelta_{x,\gamma_{1,y}}(A\cap\{\gamma(0)\in B\})\nnu(dx,dy)\\
&\overset{4}{=}\int_B\Big[ \int_{\{x\}\times\T^d} \ddelta_{x,\gamma_{1,y}}(A) d\nnu_x\Big]dx \\
&\overset{5}{=}\int_B\Big[\int_{\T^d}\ddelta_{\gamma_{1,y}}(A)\tilde\nnu_x(dy)\Big]dx.
\end{split} 
	\end{equation}
	In equality $1$, we have used that $\eeta_{0,x}$ is supported on $\{\gamma(0)=x\}$ for every $x\in\T^d$. In equality $2$, we have used that $\{\eeta_{0,x}\}$ is a disintegration of of $\eeta$ with respect to $e_0$ and $\Leb{d}$. In equality $3$, we have used Lemma \ref{lem_disint_eeta_nu}. In equality $4$, we have used that $\{\nnu_x\}$ is a disintegration of $\nnu$ with respect to $\pi_0$ and $\Leb{d}$, equation \eqref{eqn_int_disintegration}, as well as the fact that $\ddelta_{x,\gamma_{1,y}}(A\cap \{\gamma(0)\in B\})=0$ if $x\notin B$ by definition. 
In equality $5$, we have used the definition of $\ddelta_{x,\gamma_{1,y}}$ as well as the definition of $\{\tilde\nnu_x\}$. 
	As $B$ was an arbitrary Borel set in $\T^d$, there exists a set $N_A$ of vanishing Lebesgue measure such that for every $x\in \T^d-N_A$, we have
	$$\eeta_{0,x}(A)=\int_{\T^d}\ddelta_{\gamma_{1,y}}(A)\tilde\nnu_x(dy).$$
	Now define $$N:=\bigcup_{A\in \mathscr{G}}N_A,$$ 
	which is a set of vanishing Lebesgue measure. 
	Then, for every $A\in \mathscr{G}$ and every $x\in \T^d- N$, we have 
	\begin{equation}
	\eeta_{0,x}(A)=\int_{\T^d}\ddelta_{\gamma_{1,y}}(A)\tilde\nnu_x(dy). 
	\end{equation}
As $\mathscr{G}$ generates the Borel $\s$-algebra of $\Gamma$, the thesis is proved. 
\end{proof}

 \subsection{Proof of $(iii)$ of Theorem \ref{thm_main}}
 We can know conclude the proof of Theorem \ref{thm_main}. 
  \begin{proof}
  	Recall that $\eeta$ is the unique Lagrangian representation of $(1,\bb)$ from the first part of Theorem \ref{thm_main}, and that we have fixed a regular measurable selection $\{\gamma_{1,y}\}$ of integral curves of $\bb$.
  	In view of Lemma \ref{lem_disint_nu_x}, there exists a Borel family of probability measures $\{\tilde\nnu_x\}$ defined in Section \ref{sec_disint} such that 
  	\begin{equation}
  	\eeta_{0,x}=\int_{\T^d}\ddelta_{\gamma_{1,y}}\tilde\nnu_x(dy),
  	\end{equation}
  	for $\Leb{d}$-a.e. $x\in \T^d$, which is the first part the statement of part $(iii)$ of Theorem \ref{thm_main}.

  	\bigskip
  	
  	Let us now show that for the vector field $\bb_{DP}:[0,1]\times \T^2\to \R^2$ constructed by Depauw, the family probability measures $\tilde\nnu_x$ are not Dirac masses for $\Leb{2}$-a.e. $x\in \T^2$. From the Appendix, we have that $\rho^B, \;\rho^W:[0,1]\times \T^2\to \R^+$ are two bounded densities in $C([0,1];w^*-L^\infty(\T^2))$ such that:
  	\begin{enumerate} 
  	\item  $\rho^B(1,\bb_{DP})$ and $\rho^W(1,\bb_{DP})$ solve \eqref{eq_pde};
  	\item $\rho^B(0,\cdot)=1/2=\rho^W(0,\cdot)$;
  	\item $\rho^B(t,\cdot)+\rho^W(t,\cdot)=1$ for every $t\in[0,1]$;
  	\item $\supp\rho^B(1,\cdot)\cup \supp\rho^W(1,\cdot)=\T^2$;
  	\item $\supp\rho^B(1,\cdot)\cap \supp\rho^W(1,\cdot)$ is of vanishing Lebesgue measure. 
  	 \end{enumerate} 
 
 \bigskip 
 
 Let $\eeta^B$ and $\eeta^W$ be two Lagrangian representations of $\rho^B(1,\bb_{DP})$ and $\rho^W(1,\bb_{DP})$ respectively, whose existence follows from Ambrosio's superposition principle. Let $\nnu^B$ and $\nnu^W$ be two probability measures given by $\nnu^B=(e_0,e_1)_\#\eeta^B$ and $\nnu^W=(e_0,e_1)_\#\eeta^W$. Let $\{\eeta_{0,x}^B\}$ be a disintegration of $\eeta^B$ with respect to $e_0$ and $\Leb{2}$, and let $\{\nnu_x^B\}$ be a disintegration of $\nnu^B$ with respect to $\pi_0$ and $\Leb{2}$. Similarly, let $\{\eeta_{0,x}^W\}$ be a disintegration of $\eeta^W$ with respect to $e_0$ and $\Leb{2}$, and let $\{\nnu_x^W\}$ be a disintegration of $\nnu^W$ with respect to $\pi_0$ and $\Leb{2}$.
 Note that by definition of $\nnu^B$ and $\nnu^W$, we have
 $$(\pi_1)_\#\nnu^B=(e_1)_\#\eeta^B\qquad\text{and}\qquad (\pi_1)_\#\nnu^W=(e_1)_\#\eeta^W.$$ 
This clearly implies that for $\Leb{2}$-a.e. $x\in\T^2$, we have
 \begin{equation}\label{eqn_disint_nu}
  (\pi_1)_\#\nnu^B_x=(e_1)_\#\eeta_{0,x}^B \qquad\text{and} \qquad (\pi_1)_\#\nnu^W_x=(e_1)_\#\eeta_{0,x}^W.
  \end{equation}
  Therefore, we have 
  \begin{equation*}
  \begin{split} 
  \int_{\T^2}(\pi_1)_\#\nnu^B_x(\supp \rho^W(1,\cdot))dx&=\int_{\T^2}(e_1)_\#\eeta^B_x(\supp\rho^W(1,\cdot))dx\\
  &=(e_1)_\#\eeta^B(\supp \rho^W(1,\cdot))\\
  &=\int_{\supp\rho^W(1,\cdot)}\rho^B(1,x)dx\\
  &=0.
  \end{split} 
  \end{equation*}
  Similarly, we have 
  \begin{equation*}
   \int_{\T^2}(\pi_1)_\#\nnu^W_x(\supp \rho^B(1,\cdot))dx=0.
  \end{equation*}
Therefore, for $\Leb{2}$-a.e. $x\in\T^2$, in view of property (iv) above, the probability measures $(\pi_1)_\#\nnu^W_x$ and $(\pi_1)_\#\nnu^B_x$ are mutually singular. 
Also, by property (iii) above, and by uniqueness of the Lagrangian representation of $(1,\bb)$, we have that 
\begin{equation}
\eeta=\frac{1}{2}(\eeta^B+\eeta^W).
\end{equation}
 By essential uniqueness of the disintegration, we therefore have for $\Leb{2}$-a.e. $x\in \T^2$  
\begin{equation*}
\nnu_x=\frac{1}{2}(\nnu^W_x+\nnu^B_x).
\end{equation*}
Therefore, for $\Leb{2}$-a.e. $x\in\T^2$, we have
\begin{equation*}
\tilde\nnu_x=(\pi_1)_\#\nnu_x=\frac{1}{2}(((\pi_1)_\#\nnu^W_x+(\pi_1)_\#\nnu^B_x).
\end{equation*}
whereby for $\Leb{2}$-a.e. $x\in \T^2$ the probability measure $\tilde\nnu_x$ is not a Dirac mass. 
This concludes the proof of $(iii)$ of Theorem \ref{thm_main}. 
\end{proof}

\section*{Appendix}
We construct the bounded, divergence-free vector field $\bb_{DP}:[0,1]\times\T^2\to \R^2$ of Depauw from \cite{Depauw}, as well as two densities $\rho^W,\rho^B:[0,1]\times \T^2\to \R^+$ such that the vector fields $\rho^W(1,\bb_{DP})$ and $\rho^B(1,\bb_{DP}) $ solve \eqref{eq_pde}, and have the following properties:
  	\begin{enumerate} 
	\item  $\rho^B(1,\bb_{DP})$ and $\rho^W(1,\bb_{DP})$ solve \eqref{eq_pde};
	\item $\rho^B(0,\cdot)=1/2=\rho^W(0,\cdot)$;
	\item $\rho^B(t,\cdot)+\rho^W(t,\cdot)=1$ for every $t\in[0,1]$;
		\item $\supp\rho^B(1,\cdot)\cup \supp\rho^W(1,\cdot)=\T^2$;
	\item $\supp\rho^B(1,\cdot)\cap \supp\rho^W(1,\cdot)$ is of vanishing Lebesgue measure.
\end{enumerate} 
We follow closely the construction of a similar vector field given in \cite{DeLellis_Giri22}. 
\bigskip 

Introduce the following two lattices on $\mathbb R^2$, namely $\mathcal{L}^1 := \mathbb Z^2\subset \mathbb R^2$ and $\mathcal{L}^2:=\mathbb Z^2 + (\frac{1}{2}, \frac{1}{2})\subset \mathbb R^2$. To each lattice, associate a subdivision of the plane into squares, which have vertices lying in the corresponding lattices, which we denote by $\mathcal{S}^1$ and $\mathcal{S}^2$. Then consider the rescaled lattices $\mathcal{L}^1_k:= 2^{-k} \mathbb{Z}^2$ and $\mathcal{L}^2_k := (2^{-k-1},2^{-k-1})+2^{-k} \mathbb Z^2$ and the corresponding square subdivision of $\mathbb Z^2$, respectively $\mathcal{S}^1_k$ and $\mathcal{S}^2_k$. Observe that the centres of the squares $\mathcal{S}^1_k$ are elements of $\mathcal{L}^2_k$ and viceversa.


Next, define the following $2$-dimensional autonomous vector field:
\[
\ww(x) =
\begin{cases}
(0, 4x_1)^t\text{ , if }1/2 > |x_1| > |x_2| \\
(-4x_2, 0)^t\text{ , if }1/2 > |x_2| > |x_1| \\
(0, 0)^t\text{ , otherwise.} \\
\end{cases}
\]
$\ww$ is a bounded, divergence-free vector field, whose derivative is a finite matrix-valued Radon measure given by
\begin{equation*}
\begin{split} 
D\ww(x_1,x_2) = 
&\begin{pmatrix}
0 & 0 \\
4{\rm sgn}(x_1) & 0 
\end{pmatrix} 
\Leb{2}\lfloor_{\{|x_2|<|x_1|<1/2\}} +
\begin{pmatrix}
0 & -4{\rm sgn}(x_2) \\
0 & 0 
\end{pmatrix} 
\Leb{2}\lfloor_{\{|x_1|<|x_2|<1/2\}}\\
&+\begin{pmatrix}
4x_2{\rm sgn}(x_1) & -4x_2{\rm sgn}(x_2) \\
4x_1{\rm sgn}(x_1) & -4x_1{\rm sgn}(x_2)
\end{pmatrix} 
\mathscr{H}^{1}\lfloor_{\{x_1=x_2,0<|x_1|,|x_2|\leq 1/2\}}\\
\end{split} 
\end{equation*}

Periodise $\ww$ by defining $\Lambda = \{(y_1, y_2) \in \mathbb{Z}^2 : y_1 + y_2 \text{ is even}\}$ and setting 
\[
\uu(x) = \sum_{y \in \Lambda} \ww(x-y)\, .
\]

Even though $\uu$ is non-smooth, it is in $BV_{loc}(\R^2;\R^2)$. 
By the theory of regular Lagrangian flows (see for instance \cite{ambrosiocrippaedi}), there exists a unique incompressible almost everywhere defined flow $\XX$ along $\uu$ can be described explicitely. 
\begin{itemize}
	\item[(R)] The map $\XX (t,0 ,\cdot)$ is Lipschitz on each square $S$ of $\mathcal{S}^2$ and $\XX (1/2,0, \cdot)$ is a clockwise rotation of $\pi/2$ radians of the ``filled'' $S$, while it is the identity on the ``empty ones''. In particular for every $j\geq 1,$ $\XX (1/2,0, \cdot)$ maps an element of $\mathcal{S}^1_j$ rigidly onto another element of $\mathcal{S}^1_j$. For $j=1$ we can be more specific. Each $S\in \mathcal{S}^2$ is formed precisely by $4$ squares of $\mathcal{S}^1_1$: in the case of ``filled'' $S$ the $4$ squares are permuted in a $4$-cycle clockwise, while in the case of ``empty'' $S$ the $4$ squares are kept fixed.  
\end{itemize}

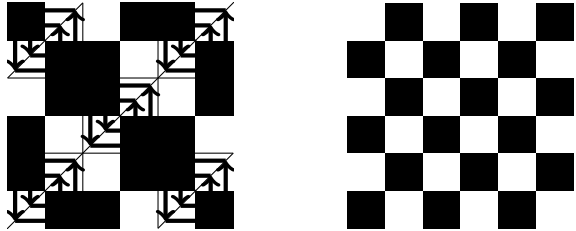
\begin{figure}[h]
	\centering
	\subfloat{
		\begin{tikzpicture}
		\clip(-1.5,-1.5) rectangle (1.5,1.5);
		\foreach \x in {-1,...,2} \foreach \y in {-1,...,2}
		{
			\pgfmathparse{mod(\x+\y,2) ? "black" : "white"}
			\edef\nu{\pgfmathresult}
			\path[fill=\nu, opacity=0.5] (\x-1,\y-1) rectangle ++ (1,1);
		}
		\draw (-1.5,0.5) -- (1.5,0.5);
		\draw (-1.5,-0.5) -- (1.5,-0.5);
		\draw (0.5, -1.5) -- (0.5, 1.5);
		\draw (-0.5, -1.5) -- (-0.5, 1.5);
		\foreach \a in {1,...,2}
		{
			\draw[ultra thick, ->] (\a/5,-\a/5) -- (\a/5,\a/5); 
			\draw[ultra thick, ->] (\a/5,\a/5) -- (-\a/5,\a/5);
			\draw[ultra thick, ->] (-\a/5,\a/5) -- (-\a/5,-\a/5);
			\draw[ultra thick, ->] (-\a/5,-\a/5) -- (\a/5,-\a/5);
		}
		\foreach \a in {1,...,2}
		{
			\draw[ultra thick, ->] (\a/5+1,-\a/5+1) -- (\a/5+1,\a/5+1); 
			\draw[ultra thick, ->] (\a/5+1,\a/5+1) -- (-\a/5+1,\a/5+1);
			\draw[ultra thick, ->] (-\a/5+1,\a/5+1) -- (-\a/5+1,-\a/5+1);
			\draw[ultra thick, ->] (-\a/5+1,-\a/5+1) -- (\a/5+1,-\a/5+1);
		}
		\foreach \a in {1,...,2}
		{
			\draw[ultra thick, ->] (\a/5+1,-\a/5-1) -- (\a/5+1,\a/5-1); 
			\draw[ultra thick, ->] (\a/5+1,\a/5-1) -- (-\a/5+1,\a/5-1);
			\draw[ultra thick, ->] (-\a/5+1,\a/5-1) -- (-\a/5+1,-\a/5-1);
			\draw[ultra thick, ->] (-\a/5+1,-\a/5-1) -- (\a/5+1,-\a/5-1);
		}
		\foreach \a in {1,...,2}
		{
			\draw[ultra thick, ->] (\a/5-1,-\a/5-1) -- (\a/5-1,\a/5-1); 
			\draw[ultra thick, ->] (\a/5-1,\a/5-1) -- (-\a/5-1,\a/5-1);
			\draw[ultra thick, ->] (-\a/5-1,\a/5-1) -- (-\a/5-1,-\a/5-1);
			\draw[ultra thick, ->] (-\a/5-1,-\a/5-1) -- (\a/5-1,-\a/5-1);
		}
		\foreach \a in {1,...,2}
		{
			\draw[ultra thick, ->] (\a/5-1,-\a/5+1) -- (\a/5-1,\a/5+1); 
			\draw[ultra thick, ->] (\a/5-1,\a/5+1) -- (-\a/5-1,\a/5+1);
			\draw[ultra thick, ->] (-\a/5-1,\a/5+1) -- (-\a/5-1,-\a/5+1);
			\draw[ultra thick, ->] (-\a/5-1,-\a/5+1) -- (\a/5-1,-\a/5+1);
		}
		\draw (-2,-2) -- (2,2);
		\draw (-2,2) -- (2,-2);
		\draw (-1.5,0.5) -- (-0.5, 1.5);
		\draw (1.5,0.5) -- (0.5,1.5);
		\draw (1.5,-0.5) -- (0.5, -1.5);
		\draw (-1.5,-0.5) -- (-0.5,-1.5);
		\end{tikzpicture}%
	}
	\qquad
	\qquad
	\subfloat
	{
		\begin{tikzpicture}[scale=0.5]
		\foreach \x in {-2,...,3} \foreach \y in {-2,...,3}
		{
			\pgfmathparse{mod(\x+\y,2) ? "white" : "black"}
			\edef\nu{\pgfmathresult}
			\path[fill=\nu, opacity=0.5] (\x-1,\y-1) rectangle ++ (1,1);
		}
		\end{tikzpicture}%
	}
	\caption{Action of the flow of $\uu$ from $t=0$ to $t=1/2$. The shaded region denotes the set $\{\rho^B=1\}$. The figure is from \cite{DeLellis_Giri22}.}
\end{figure}

Let $\rho^B:[1/2,1]\times \R^2\to \R^+$ be the unique  density such that $\rho^B(1,\uu)$ solves \eqref{eq_pde} and $\rho^B(1,\cdot)=\lfloor{x_1}\rfloor /2+ \lfloor{x_2}\rfloor/2 \ mod \ 2=:\bar\rho^B$. Then, we have the following formula  $\XX(t,0,\cdot)_\#\bar\rho^B\Leb{d}=\rho^B(t,x)\Leb{d}$. 
Using property (R), we have 
\begin{equation}\label{e:refining}
\rho^B(1/2, x) = 1 - \bar\rho^B(2x) .
\end{equation}

We define $\bb_{DP}:[0,1]\times\R^2\to\R^2$ as follows. Set $\bb_{DP}(t, x) = \uu(x)$ for $1/2<t\leq1$ and $\bb_{DP}(t, x) = \uu(2^k x)$ for $1/2^{k+1}<t\leq1/2^{k}$. Let $\rho^B:[0,1]\times \R^2\to \R^+$ be the unique  density such that $\rho^B(1,\uu)$ solves \eqref{eq_pde} with $\rho^B(0,\cdot)=\lfloor{x_1}\rfloor/2 + \lfloor{x_2}\rfloor/2 \ mod \ 2=:\bar\rho^B$
Moreover, using recursively the appropriately scaled version of \eqref{e:refining}, we can check that 
\begin{equation*} 
\rho^B(1/2^{k}, x) = \bar\rho^B (2^{k} x) \quad\text{for $k$ even,}\qquad \rho^B (1/2^{k}, x) =1- \bar\rho^B (2^{k} x)\quad\text{ for $k$ odd. }
\end{equation*} 

Define the density $\rho^W(t,x):=1-\rho^B(t,x)$. Then $\rho^W(1,\bb_{DP})$ also solves \eqref{eq_pde}, by linearity. As the construction we have performed is $\Z^2$-periodic, we may consider $\bb_{DP}$, $\rho^W,$ and $\rho^B$ to be defined on $[0,1]\times \T^2$. 
Properties (i)-(v) follow directly from the construction. 

\bibliographystyle{alpha}
\bibliography{bibliografia}
\end{document}